\documentclass[10pt]{amsart}
\usepackage{amsmath}
\usepackage{amsthm}
\usepackage{color}
\usepackage{amscd}
\usepackage{amsfonts}
\usepackage{graphicx}
\usepackage{epsfig}
\usepackage{amssymb,latexsym}
\usepackage{bm}

\newcommand{\adj}{\operatorname{adj}}

\def\bbu{{\boldsymbol{u}}}

\def\bbU{{\boldsymbol{U}}}
\def\bu{\mathbf{u}}

\def\bv{\mathbf{v}}

\def\bU{\mathbf{U}}
\def\cU{\mathcal{U}}
\def\bV{\mathbf{V}}
\def\bbx{{\boldsymbol{x}}}
\def\bx{\mathbf{x}}

\def\bby{{\boldsymbol{y}}}
\def\by{\mathbf{y}}

\def\X{\mathbf{X}}
\def\bbf{\mathbf{f}}
\def\bbg{\mathbf{g}}
\def\R{\mathbb{R}}
\def\bD{\mathbf{D}}

\def\F{\mathbf F}
\def\G{\mathbf G}

\def\cM{\mathcal M}
\def\cI{\mathcal I}
\def\bQ{\mathbf Q}
\def\cD{\mathcal D}
\def\A{\mathbf A}

\def\cC{\mathcal C}

\def\id{\mathbf{id}}
\def\bP{\mathbf{P}}

\def\bbS{\mathbb S}
\def\bcero{\mathbf 0}

\numberwithin{equation}{section}

\newtheorem{theorem}{Theorem}[section]
\newtheorem{corollary}[theorem]{Corollary}
\newtheorem{definition}{Definition}[section]
\newtheorem{lemma}[theorem]{Lemma} 
\newtheorem{proposition}[theorem]{Proposition}
\newtheorem{problem}[theorem]{Problem}

\newtheorem{remark}[theorem]{Remark}
\newtheorem{conjecture}[theorem]{Conjecture}

\begin{document}

\title{Invariant Hulls and Geometric Variational Principles}
\author{Pablo Pedregal}
\date{} 
\thanks{Departmento de Matemáticas, Universidad de Castilla-La Mancha, 13071 Ciudad Real, SPAIN. Supported by grants 
PID2020-116207GB-I00, and  SBPLY/19/180501/000110}
\begin{abstract}
We investigate functionals defined on manifolds through parameterizations. If they are to be meaningful, from a geometrical viewpoint, they ought to be invariant under reparameterizations. Standard, local, integral functionals with this invariance property are well-known. We would like to focus though on the passage from a given arbitrary functional to its invariant realization or invariant hull through the use of inner-variations, much in the same way as with the convex or quasiconvex hulls of integrands in the vector Calculus of Variations. These two processes are, however, very different in nature. After examining some basic, interesting, general properties about the mutual relationship between a functional and its invariant realization, we deal with the one dimensional case to gain some initial familiarity with such a transformation and calculations, before proceeding to the higher dimensional situation. As one would anticipate, explicit computations in the latter are much harder to perform, if not impossible, as one is to work with vector variational problems. In particular, we are able to reach some  modest conclusion about the volume functional of a piece of a manifold in the general $N$-dimensional situation, especially in the two-dimensional case $N=2$. Various problems and conjectures are stated along the way. 
\end{abstract}

\maketitle


\vspace{10pt} {\bf AMS(MOS) subject classifications.} 49Q10, 58E12, 49Q05.

\section{Introduction}
We would like to look at functionals 
$$
I(\bbu):W^{1, p}(\Omega; \R^m)\to\R
$$
where $\Omega\subset\R^N$ is some model domain, like the unit ball or square, and $W^{1, p}(\Omega; \R^m)$, $p\ge1$, is the standard Sobolev space. 
We want to focus on the case $N<m$, which can be appropriately labeled the geometric case, where feasible maps 
$$
\bbu\in W^{1, p}(\Omega; \R^m)
$$ 
can be interpreted as parameterizations of pieces of $N$-dimensional manifolds embedded or immersed in $\R^m$. These variational problems can also be declared as low-dimensional vector variational problem to stress the fact that $N<m$. 
One must insist in that such objects should be regular or geometric in the sense that the jacobian matrix 
$$
\nabla\bbu(\bbx)\in\R^{m\times N}
$$ 
should have maximum rank $N$ at every point $\bbx\in\Omega$,
and, in particular, induce a well-defined orientation of the associated manifold. 
It is important to turn our attention to reparameterizations of the same underlying geometric object $\bbu(\Omega)\subset\R^m$ for each $\bbu$. In fact, our emphasis is on functionals that are parameterization-invariant, as these would be, rightly so, the ones accepted for geometric analysis. 

It is well-known (\cite{giahil}, \cite{iwmaronn}) that local, integral functionals given by a density 
$$
W(\bu, \bU):\R^m\times\R^{m\times N},\quad I(\bbu)=\int_\Omega W(\bbu(\bbx), \nabla\bbu(\bbx))\,d\bbx,
$$
are invariant with respect to parametrizations if 
$$
W(\bu, \bU\F)=\det\F\, W(\bu, \bU),\quad (\bu, \bU)\in\R^m\times\R^{m\times N}, \F\in\R^{N\times N}, \det\F>0,
$$
which in turn is equivalent to the existence of a density
$$
w(\bu, \bV):\R^m\times\R^n\to\R,\quad n=\binom mN
$$
positively homogeneous of degree one on the variable $\bV$, such that
\begin{equation}\label{representation}
W(\bu, \bU)=w(\bu, \bU_1\wedge\bU_2\wedge\dots\wedge\bU_N),\quad \bU=\begin{pmatrix}\bU_1&\bU_2&\dots&\bU_N\end{pmatrix}.
\end{equation}
As a matter of fact, variational principles associated with the functional
$$
I(\bbu)=\int_\Omega w\left(\bbu, \frac{\partial\bbu}{\partial\bbx_1}\wedge\frac{\partial\bbu}{\partial\bbx_2}\wedge\dots\wedge\frac{\partial\bbu}{\partial\bbx_N}\right)\,d\bbx
$$
are indeed quite delicate. Just think about the particular example
\begin{equation}\label{minimals}
A(\bbu)=\int_\Omega \left|\frac{\partial\bbu}{\partial\bbx_1}\wedge\frac{\partial\bbu}{\partial\bbx_2}\right|\,d\bbx,\quad \Omega\subset\R^2, \bbu(\bbx):\Omega\to\R^3,
\end{equation}
whose minimizers, possibly under suitable boundary conditions, correspond to minimal area solutions, and the overwhelming amount of deep and fundamental work that has stirred over the years. 
We would like to highlight the important case in which $m=N+1$, $n=N+1$, and
$$
w(\bu, \bv):\R^{N+1}\times\R^{N+1}\to\R
$$ 
is positively-homogeneous of degree one in $\bv$. 

To be specific, though much more general situations may be considered, one fundamental problem is the following.
\begin{problem}\label{primeroo}
Find densities $w(\bv):\R^3\to\R$, positively-homogeneous of degree one, and suitable classes $\cU$ of maps in $H^1(\Omega; \R^3)$ where $\Omega\subset\R^2$ is a model domain, such that the variational problem
$$
\hbox{Minimize in }\bbu\in\cU:\quad \int_\Omega w\left(\frac{\partial\bbu}{\partial\bbx_1}(\bbx)\wedge\frac{\partial\bbu}{\partial\bbx_2}(\bbx)\right)\,d\bbx
$$
admits minimizers.
\end{problem}
As far as we can tell, the only situation where Problem \ref{primeroo} can be shown to be solvable is the classical minimal surface problem \eqref{minimals} for a suitable class $\cU$ leaving monotonically invariant the boundary $\partial\Omega$, and three different, prescribed points at the boundary $\partial\Omega$, unchanged. This is classical. Check the recent encyclopedic work \cite{dihilsau}, or the more classical reference \cite{giusti}.  As far as we can tell, the answer is not known even for $w$ the $p$th-norm of vectors in $\R^3$, $p\neq2$. The development of the BV theory of area-minimizing hypersurfaces is, in this regard, of paramount importance (see por instance \cite{delederoghi}, \cite{giusti}, \cite{harrisonpugh}, \cite{song}).

Our viewpoint is, on the one hand, broader in the sense that we would like to let more general functionals enter into our discussion; but on the other, it is less sophisticated and more straightforward than that in the  references at the end of the previous paragraph. Our basic driving principle has a two-fold motivation, which in simple, general terms can be stated as follows:
\begin{enumerate}
\item Starting from arbitrary functionals, one can produce, in a natural way, invariant, associated functionals (invariant hulls) by minimizing on suitable classes of changes of variables.
\item Minimizers for functionals and their invariant hulls must coincide over suitable classes of invariant mappings. 
\end{enumerate}
More specifically, our program proceeds in various steps.
\begin{enumerate}
\item Start with a typical local integral functional of the form 
$$
I(\bbu)=\int_\Omega W(\bbx, \bbu(\bbx), \nabla\bbu(\bbx))\,d\bbx,\quad \bbu\in W^{1, p}(\Omega; \R^m), \Omega\subset\R^N, N<m,
$$
or, for the sake of definiteness, focus on the situation
\begin{equation}\label{funcionala}
I(\bbu)=\int_\Omega W(\nabla\bbu(\bbx))\,d\bbx,\quad W(\bU):\R^{m\times N}\to\R,
\end{equation}
and eventually assume the necessary hypotheses on $W$ and the underlying class of competing maps $\bbu(\bbx)$, to ensure existence of minimizers.
\item We will define the invariant hull $I_i(\bbu)$ of the functional $I$, by minimizing on a suitable class of changes of variables $\bbx=\Phi(\bby)$ in $\Omega$
$$
I_i(\bbu)=\inf_{\Phi}\int_\Omega W(\nabla\bbu_{\Phi}(\bbx))\,d\bbx, \quad \bbu_\Phi(\bbx)=\bbu(\Phi^{-1}(\bbx)).
$$
This new functional $I_i$ is well-defined in most cases of interest, and, by construction, is parameterization-invariant. A different issue is if it admits a local, integral representation, and if so, if there are minimizers for it in the same class of mappings. We will see that this is not always so, in spite of the fact that the initial functional $I$ is of the standard form \eqref{funcionala}. 
\item At any rate, every minimizer of $I$ over a suitably restricted invariant subclass of $W^{1, p}(\Omega; \R^m)$, will be a minimizer of $I_i$ over the same class. 
\end{enumerate}

After going over some general material related to the passage from $I$ to $I_i$, and some interesting properties of $I_i$, we will focus on the one-dimensional case $N=1$ for which some explicit, simple computations are possible. This will serve as an initial interesting training ground to better appreciate the nature of the operation $I\mapsto I_i$, and to have at our disposal some simple explicit examples. Then we will focus on the much more difficult multidimensional case $N>1$ to treat the particularly interesting example 
\begin{equation}\label{especial}
I(\bbu)=\int_\Omega \frac1{N^{N/2}}|\nabla\bbu(\bbx)|^N\,d\bbx,
\end{equation}
and see how far we can go in finding its invariant realization $I_i$, depending on dimension $N$. Functionals of the form
$$
\int_\Omega W(\nabla\bbu(\bbx))\,d\bbx,\quad W(\F):\R^{m\times N}\to\R,
$$
where $W$ is homogeneous of degree $N$, seem particularly interesting beyond example \eqref{especial}. 
In general terms, for power functionals
$$
I_p(\bbu)=\int_\Omega|\nabla\bbu(\bbx)|^p\,d\bbx,\quad p>0,
$$
not much can be said at this point in the multidimensional situation. 

Note that by insisting in that $\Omega$ is a ball or a square, we are consequently limiting the discussion in some very fundamental way since the topological class of the image manifold is being restricted. 

\section{Some initial concepts and basic properties}
This section gathers a few elementary facts that do not require any particular form of functionals, and whose proofs are completely elementary; yet they are worth to bear in mind. 

\begin{definition}\label{invarianza}
\begin{enumerate}
\item We designate by $\cD=\cD(\Omega)$ the class of all smooth, positively-oriented ($\cC^\infty$)-diffeomorhisms $\Phi$ of $\overline\Omega$ onto itself. 
We also put $\cD_\id=\cD_\id(\Omega)$ for the subclass of $\cD$ of diffeomorphisms that are identical to the identity on $\partial\Omega$, and regard all possible invariant subclasses 
\begin{equation}\label{subclase}
\cD_\id\subset \cD_0\subset\cD
\end{equation}
in the sense
$$
\Phi_1\circ\Phi_2\in\cD_0\hbox{ if } \Phi_i\in\cD_0, i=1, 2.
$$
Obviously $\cD_\id$ and $\cD$ are invariant.
\item For an invariant subclass $\cD_0$ in \eqref{subclase}, we say that a collection of maps
$$
\cU\subset W^{1, p}(\Omega; \R^m)
$$
is invariant under $\cD_0$ if it is invariant under the action 
\begin{equation}\label{accion}
\bbu_\Phi(\bbx)=\bbu(\Phi^{-1}(\bbx)),
\end{equation}
that is to say $\bbu_\Phi\in\cU$ whenever $\bbu\in\cU$ and $\Phi\in\cD_0$.
\item We say that a certain functional $I$ is invariant or geometric, if  
\begin{equation}\label{invarianza}
I(\bbu_\Phi)=I(\bbu),\quad \bbu\in W^{1, p}(\Omega; \R^m), \Phi\in\cD;
\end{equation}
and more specifically that it is invariant under a given invariant class $\cD_0$, if \eqref{invarianza} only holds for $\Phi\in\cD_0$.
\end{enumerate}
\end{definition}

Note the following. 
\begin{enumerate}
\item The condition 
$$
\Phi(\bbx)=\bbx,\quad \bbx\in\partial\Omega,
$$
is much more restrictive that the one accepted for $\cD$; in fact, elements of $\cD$ ought to maintain  $\partial\Omega$ invariant
$$
\Phi(\partial\Omega)=\partial\Omega,\quad \det\nabla\Phi(\bbx)>0\hbox{ a.e. in }\Omega.
$$
Sometimes boundary conditions associated with $\cD$ are called ``frictionless" (\cite{iwmaronn}, and references therein). 
Typically, subclasses $\cD_0$ in \eqref{subclase} are determined by specifying the subset of $\partial\Omega$ that is to be maintained unchanged by elements $\Phi$ of $\cD_0$, among possibly further restrictions. 
\item The map $\bbu_\Phi$ belongs to $W^{1, p}(\Omega; \R^m)$
whenever
$$
\bbu\in W^{1, p}(\Omega; \R^m), \Phi\in\cD.
$$
\item All elements in $\cD$ are legitimate changes of variables in $\Omega$.
\item Functionals are not assumed necessarily to be standard, local, integral functionals at this stage.
\end{enumerate}

Suppose we are given 
a general functional, not necessarily an integral functional, 
$$
I(\bbu):W^{1, p}(\Omega; \R^m)\to\R
$$ 
which is assumed to be well-defined over a class $\cU\subset W^{1, p}(\Omega; \R^m)$ of mappings, invariant with respect to some $\cD_0$. We will assume that such an invariant class $\cD_0$ has been appropriately selected, and refer to it in all of our manipulations when we simply use the term invariant. 

There is a natural way to produce invariant functionals from arbitrary examples by minimization over $\cD_0$
\begin{equation}\label{funcionali}
I_i(\bbu)=\inf_{\Phi\in\cD_0}I(\bbu_\Phi),
\end{equation}
where $\bbu_\Phi$ is given in \eqref{accion}. 

\begin{definition}
Given an arbitrary functional 
$$
I(\bbu):W^{1, p}(\Omega; \R^m)\to\R,
$$ 
we call 
$$
I_i(\bbu):W^{1, p}(\Omega; \R^m)\to\R
$$ 
in \eqref{funcionali}, its invariant realization with respect to $\cD_0$.
\end{definition}

It is elementary to realize that the invariant realization of every functional is invariant (under the same class $\cD_0$), and hence we have a mechanism to produce all of the possible invariant functionals. However, given that trivial (constant) functionals are obviously invariant, it may happen that sometimes the functional $I_i$ might turn out to be trivial even though $I$ could be quite meaningful. In the same vein, the passage $I\mapsto I_i$ may change the nature of the functional. For instance, $I$ could be an integral functional, but $I_i$ may not be so. 

We establish some of these basic properties formally for future reference. 

\begin{proposition}
For every functional $I:\cU\to\R$, its invariant or geometric version $I_i:\cU\to\R$ defined through \eqref{funcionali} is invariant, and $I_i\le I$. 
\end{proposition}
As indicated, this statement does not require a proof. It is a consequence of its own definition. The functional $I_i$ can be given a parallel characterization in terms of invariant functionals.

\begin{proposition}\label{supremo}
For $I:\cU\to\R$ as above, 
$$
I_i=\sup\{E:\cU\to\R: E\le I, E,\hbox{ invariant}\}.
$$
\end{proposition}
\begin{proof}
Put, for the time being, 
$$
I_g=\sup\{E:\cU\to\R: E\le I, E,\hbox{ invariant}\}.
$$
If $E$ is invariant and $E\le I$, then it is clear that $E=E_i\le I_i$, and so $I_g\le I_i$. Conversely, since $I_i$ is invariant and $I_i\le I$, we should also have $I_i\le I_g$. 
\end{proof}

\begin{corollary}\label{supinf}
Suppose that a functional $E$ is invariant, and $E\le I$. Assume that for each feasible $\bbu$ given, there is a sequence $\{\Phi_j\}\in\cD(\Omega)$ such that $I(\bbu_{\Phi_j})\to E(\bbu)$. Then $I_i\equiv E$, and $\{\Phi_j\}$ is minimizing for \eqref{funcionali}. 
\end{corollary}
\begin{proof}
The existence of the sequence $\{\Phi_j\}$, for arbitrary $\bbu$, with the claimed property, implies that $I_i\le E$. But, because $E$ is invariant and $E\le I$, by the preceding result, $I_i\ge E$. In particular, sequences $\{\Phi_j\}$, for each $\bbu$, become minimizing for \eqref{funcionali}.
\end{proof}

Our main basic result is concerned with the interplay between the optimization problems for both $I$ and $I_i$. 
\begin{proposition}\label{existencia}
We always have
$$
\inf_{\bbu\in\cU}I(\bbu)=\inf_{\bbu\in\cU}I_i(\bbu).
$$
Moreover, if we know that
$$
\inf_{\bbu\in\cU}I(\bbu)=\min_{\bbu\in\cU}I(\bbu)=I(\bbu_0),\quad \bbu_0\in\cU,
$$
then $\bbu_0$ becomes a minimizer for $I_i$ in $\cU$ as well
$$
I_i(\bbu_0)=\min_{\bbu\in\cU}I_i(\bbu).
$$
\end{proposition}
\begin{proof}
It is elementary to realize that
$$
\inf_{\bbu\in\cU}I_i(\bbu)=\inf_{\bbu\in\cU}\inf_{\Phi\in\cD}I(\bbu_\Phi)=\inf_{\bbu\in\cU}I(\bbu),
$$
since $\cU$ is assumed to be invariant. If the infimum of $I$ over $\cU$ is attained at some $\bbu_0\in\cU$, then
$$
I_i(\bbu_0)\le I(\bbu_0)=\inf_{\bbu\in\cU}I(\bbu)=\inf_{\bbu\in\cU}I_i(\bbu)\le I_i(\bbu_0),
$$
and our conclusion follows.
\end{proof}
This proposition is the principal method to tackle the existence of minimizers for geometric variational principles. It somehow amounts to the inverse process of the one we are describing here.

\begin{problem}
Given a relevant invariant functional $I$, defined on a certain family $\cM$ of manifolds, find another one $\cI$, convex and coercive in such a way that admits minimizers $\bbu$ in $\cM$, designed so that $\cI_i=I$. In this case $\bbu$ becomes a minimizer for $I$ in $\cM$ as well. 
\end{problem}

Another elementary, interesting consequence lets determine many other invariant realizations. 
\begin{proposition}\label{intermedio}
Suppose we have $E\le I$, and $I_i\le E$. Then $E_i=I_i$.
\end{proposition}

\section{The one-dimensional case: curves}\label{onedi}
For the sake of illustration and to gain some insight on the nature of the passage from a given arbitrary functional to its invariant realization, we pay attention in this section to the one-dimensional case, i.e. to the case of curves in $\R^m$. In this situation $\Omega$ will be taken to be the unit interval. We expect that in this situation some explicit calculations may be possible. 

We will always start with a typical integral functional 
\begin{equation}\label{basicouno}
I(\bu)=\int_0^1 W(\bu'(t))\,dt,
\end{equation}
depending on curves 
$$
\bu(t):[0, 1]\to\R^m,\quad  W(\bv):\R^m\to\R
$$
and with an integrand just depending on the derivative $\bv=\bu'$, and see how to calculate its invariant realization. Note that in this case, after a natural change of variables, we formally find
$$
\int_0^1 W(\bu'_\phi(t))\,dt=\int_0^1 \phi'(t)\,W\left(\frac1{\phi'(t)}\bu'(t)\right)\,dt.
$$
For each fixed path $\bU(=\bu')$, we need to face the optimization problem
$$
\hbox{Minimize in }\phi:\quad \int_0^1 \phi'(t)W\left(\frac1{\phi'(t)}\bU(t)\right)\,dt
$$
subjected to constraints
\begin{equation}\label{condiciones1d}
\phi(0)=0, \phi(1)=1, \quad \phi'>0.
\end{equation}
We realize that the existence of a minimizer for such a problem in an appropriate Sobolev class depends on the following elementary concept. 
\begin{definition}
A continuous function 
$$
W(\bv):\R^m\to\R
$$ 
is said to be radially convex and smooth if the sections
\begin{equation}\label{secciones}
r\in(0, \infty)\mapsto  r\,W\left(\frac1r\bx\right),\quad \bx\in\R^m\setminus\{\bcero\},
\end{equation}
as functions of the single variable $r$, are (strictly) convex and smooth for all such $\bx$.
\end{definition}
A basic result follows.

\begin{theorem}\label{dimuno}
Suppose that the function 
$$
W(\bv):\R^m\to\R
$$ 
is radially strictly convex and smooth, and let $I(\bu)$ be the corresponding functional \eqref{basicouno}. Then
$$
I_i(\bu)=\int_0^1\bbf(c(\bu), \bu'(t))\,W\left(\frac1{\bbf(c(\bu), \bu'(t))}\bu'(t)\right)\,dt
$$
where the function $\bbf(c, \bx)$ is the inverse of
$$
r\in(0, \infty)\mapsto \bbg(r, \bx)\equiv W\left(\frac1r\bx\right)-\frac1r\nabla W\left(\frac1r\bx\right)\cdot\bx,
$$
for fixed $\bx$, and the (constant) functional $c=c(\bu)$ is determined through the condition
\begin{equation}\label{condicion}
\int_0^1\bbf(c(\bu), \bu'(t))\,dt=1.
\end{equation}
\end{theorem}
\begin{proof}
According to our definition of $I_i$, we need to examine, for each feasible path $\bu(t)$, the variational problem
$$
\hbox{Minimize in }\phi(t): \int_0^1 W\left(\frac1{\phi'(t)}\bu'(t)\right)\phi'(t)\,dt
$$
under constraints \eqref{condiciones1d}. We do not know a priori if there is a minimizer for such a problem based on the direct method, since we cannot count on coercivity. 
We can, however, examine the corresponding Euler-Lagrange equation and check if it admits a solution; and realize afterwards, through convexity in the standard manner, that it is a minimizer of the problem. 

We are thus led, taking into account the notation above, to the boundary value problem
$$
-[\bbg(\phi'(t), \bu'(t))]'=0\hbox{ in }(0, 1),\quad \phi(0)=0, \phi(1)=1.
$$
We conclude that 
$$
\bbg(\phi'(t), \bu'(t))=c(\bu)\quad\hbox{ or } \quad\bbf(c(\bu), \bu'(t))=\phi'(t),
$$
with $c=c(\bu)$, a constant (with respect to $t$). We note that for each fixed vector $\bx$, the function 
$$
r\mapsto\bbg(r, \bx)
$$ 
has a non-negative derivative due to the radial strict convexity of $W$, and so it is monotone (for positive $r$). After all, $\bbg(r, \bx)$ is the first derivative of the sections in \eqref{secciones}. 
All that remains to argue is that the functional $c(\bu)$ is to be determined by the condition
$$
\int_0^1\bbf(c(\bu), \bu'(t))\,dt=\int_0^1 \phi'(t)\,dt=1.
$$
\end{proof}

With the help of this theorem, we can perform some interesting explicit calculations. We will put
$$
\overline W(\bx, r)=r\, W\left(\frac1r\bx\right),\quad \bx\in\R^m, r>0,
$$
and, for a fixed path $\bu$, 
$$
w(t, r)=r W\left(\frac1r\bu'(t)\right),
$$
so that we are interested in finding the minimum of the variational problem
$$
\hbox{Minimize in }\phi(t):\quad \int_0^1 w(t, \phi'(t))\,dt
$$
under 
$$
\phi(0)=0, \phi(1)=1,\quad \phi'>0.
$$ 

We explore next some explicit examples.

\subsection{Some explicit examples}
The first example is mandatory
$$
W(\bu):\R^m\to\R,\quad W(\bu)=|\bu|.
$$
In this case
$$
\overline W(\bx, r)=W(\bx)
$$
for all $\bx\in\R^m$ and $r>0$, and consequently
$$
w(t, r)=W(\bu'(t))
$$
is independent of $r$. The corresponding variational problem becomes trivial since
$$
\int_0^1w(t, \phi'(t))\,dt=\int_0^1|\bu'(t)|\,dt
$$
for every feasible $\phi$. This is not surprising since, after all, the initial variational problem yields the length of the image curve, and this is invariant under reparameterizations. Note that exactly the same calculations hold for every function $W(\bu)$ which is positively-homogeneous of degree one. 

More interesting is the quadratic case
$$
W(\bu):\R^m\to\R,\quad W(\bu)=\frac12|\bu|^2,
$$
which is radially coercive, and strictly convex. According to Theorem \ref{dimuno}, we should care about the inverse of the function
$$
\bbg(r, \bx)=\frac1{2r^2}|\bx|^2-\frac1{r^2}|\bx|^2=-\frac1{2r^2}|\bx|^2,
$$
with respect to the variable $r$ for fixed $\bx$, which is
$$
\bbf(c, \bx)=\frac{|\bx|}{\sqrt{-2c}}. 
$$
Condition \eqref{condicion} becomes
$$
2c=-\left(\int_0^1|\bu'(t)|\,dt\right)^2,\quad \bbf(c(\bu), \bu')=\frac1{\int_0^1|\bu'(t)|\,dt}|\bu'|,
$$
and the corresponding invariant realization, after a few manipulations, becomes
$$
I_i(\bu)=\frac12\left(\int_0^1|\bu'(t)|\,dt\right)^2.
$$
Note that this is not a typical integral functional despite the fact that the initial one was.

Let us explore now
$$
W(\bu)=\frac1p|\bu|^p,\quad 0<p,
$$
and distinguish the two ranges $0<p<1$, $1<p$. The second case is similar to the quadratic one and one finds
$$
I_i(\bu)=\frac1p\left(\int_0^1|\bu'(t)|\,dt\right)^p.
$$
The case $0<p<1$ is drastically distinct, because
$$
w(t, r)=r^{1-p}|\bu'(t)|^p,
$$
is no longer convex, and yet for every smooth path $\bu$, the infimum of
$$
\int_0^1 \psi'(t)^{1-p}\,|\bu'(t)|^p\,dt.
$$
in $\psi$ vanishes. It suffices to take the sequence $\psi_j(t)=t^{j+1}$. Computations are elementary. In this case, the invariant realization $I_i(\bu)\equiv0$ is trivial. 

An additional dependence of $W$ on $\bu$ is essentially the same as above, i.e.
$$
W=W(\bu(t), \bu'(t)).
$$
In this case, for fixed $\bu(t)$, the optimization problem in changes of variables $\phi(t)$ would be
$$
\hbox{Minimize in }\phi(t):\quad \int_0^1 W\left(\bu(t), \frac1{\phi'(t)}\bu'(t)\right)\phi'(t)\,dt
$$
under $\phi(0)=0$, $\phi(1)=1$, and both the functions 
$$
\bbg(r, \by, \bx)=W\left(\by, \frac1r\bx\right)-\frac1r\nabla_\bx W\left(\by, \frac1r\bx\right)\cdot\bx
$$
and its inverse $\bbf(r, \by, \bx)$ with respect to $r$ would show an additional dependence on the variable $\by$. But other than this point, everything else is just like the previous situation. 

\section{The local, integral case in higher dimension}
We would like to explore the situation in which we start out with a typical integral functional
\begin{equation}\label{primero}
I(\bu)=\int_\Omega W(\nabla\bu(\bbx))\,d\bbx
\end{equation}
for a certain, continuous integrand
$$
W(\F):\R^{m\times N}\to\R,
$$
and invariant class $\cD_0$ as in Definition \ref{invarianza}.
To begin with, we will assume that feasible maps 
$$
\bu(\bbx):\Omega\subset\R^N\to\R^m,\quad N<m,
$$
are smooth. Let $\Phi\in\cD_0$. It is elementary to argue that
$$
\int_\Omega W(\nabla\bu_\Phi(\bby))\,d\bby=\int_\Omega \det\nabla\Phi(\bbx) W\left(\frac1{\det\nabla\Phi(\bbx)}\nabla\bu(\bbx)\adj\nabla\Phi(\bbx)^T\right)\,d\bbx,
$$
through the change of variables $\bbx=\Phi^{-1}(\bby)$ and the standard formula
$$
\X^{-1}=\frac1{\det\X}\adj\X^T,
$$
valid for non-singular, square matrices, 
and so
\begin{equation}\label{invarianthull}
I_i(\bbu)=\inf_{\Phi\in\cD_0}\int_\Omega \overline W(\nabla\Phi(\bbx), \nabla\bu(\bbx))\,d\bbx,
\end{equation}
for
\begin{equation}\label{dosvar}
\overline W(\X, \F)=\det\X\, W\left(\F\X^{-1}\right)=\det\X\, W\left(\frac1{\det\X}\F\adj\X^T\right).
\end{equation}

In this way, we clearly see that computing $I_i(\bu)$ amounts to solving a non-homogeneous, vector, variational problem of the same kind considered in hyper-elasticity (\cite{ciarlet}) for an inhomogeneous integrand of the form
$$
\tilde W(\bbx, \X)=\overline W(\X, \nabla\bu(\bbx)),
$$
under suitable boundary conditions depending on the invariant class $\cD_0$ considered. 
Since variations of the form given in \eqref{accion} are classic inner-variations for the functional $I$, it is natural to introduce the following (\cite{pedregalelas}).
\begin{definition}\label{inner}
Let
$$
W(\F):\R^{m\times N}\to\R
$$ 
be a continuous integrand. 
\begin{enumerate}
\item $W$ is said to be  inner-quasiaffine if the corresponding $\overline W(\X, \F)$ in \eqref{dosvar} is quasi-affine in $\X$ for every fixed $\F$.
\item $W$ is said to be inner-quasiconvex if
\begin{equation}\label{modw}
\overline W(\X, \F)=\begin{cases}\det\X W(\F\X^{-1}),& \det\X>0,\\+\infty,& \det\X\le0,\end{cases}
\end{equation}
is quasiconvex in $\X$ for every fixed $\F$.
\item $W$ is said to be inner-polyconvex if $\overline W(\X, \F)$ is polyconvex in $\X$ for every fixed $\F$.
\end{enumerate}
\end{definition}
Notice that if elements of $\cD_0$ would be demanded to comply with the boundary condition $\Phi=\id$ on $\partial\Omega$, then we could have defined $W$ as inner-quasiconvex if $\overline W(\X, \F)$ is quasiconvex at $\id$ in $\X$ for every $\F$.  

The following basic fact justifies, to some extent, the previous definition.
\begin{proposition}
Assume the integrand $W$ for $I$ in \eqref{primero} is quasiconvex (in the usual, vector sense), respectively quasi-affine. Then $W$ must be inner-quasiconvex, respectively inner-quasiaffine.
\end{proposition}
\begin{proof}
Let $\F, \X$ be given constant $m\times N$-,  and  $N\times N$-matrices, respectively, with $\det\X>0$. Let 
$$
\Phi(\bby): \Omega\to\R^N,\quad \det(\X+\nabla\Phi(\bby))>0\hbox{ a.e. }\bby\in\Omega,
$$ 
be a smooth, compactly-supported map . Then
$$
\bbx=\bbu^{-1}(\bby)=\X\bby+\Phi(\bby)
$$
is a valid change of variables by the well-known results on injectivity in \cite{ball}, \cite{pourciau}, or \cite{ciarlet}, and the integral
$$
\int_\Omega \overline W(\X+\nabla\Phi(\bby), \F)\,d\bby=\int_\Omega\det(X+\nabla\Phi(\bby))W\left(\F(\X+\nabla\Phi(\bby))^{-1}\right)\,d\bby
$$
is written as
$$
\int_{\X(\Omega)} W(\F\nabla\bbu(\bbx))\,d\bbx.
$$
Because $\det\X>0$, the image domain $\X(\Omega)$ is also a regular domain. On the other hand, the map
$$
\bbU(\bbx)=\F\bbu(\bbx)
$$
is Lipschitz and such that 
$$
\left.\bbU\right|_{\partial\Omega}=\left.\F\X^{-1}\bbx\right|_{\partial\Omega},
$$
and, hence, by quasiconvexity,
\begin{align}
\int_{\X(\Omega)} W(\F\nabla\bbu(\bbx))\,d\bbx=&\int_{\X(\Omega)} W(\nabla\bbU(\bbx))\,d\bbx\nonumber\\
\ge&|\X(\Omega)| W(\F\X^{-1})\nonumber\\
=&|\Omega|\det\X W(\F \X^{-1})\nonumber\\
=&|\Omega|\overline W(\X, \F).\nonumber
\end{align}
Putting everything together, we see that 
$$
\int_\Omega \overline W(\X+\nabla\Phi(\bby), \F)\,d\bby\ge |\Omega|\overline W(\X, \F),
$$
as desired.
\end{proof}

Concerning the  third part of Definition \ref{inner} about poly-convexity, check \cite{pedregal}.

Even if one could count on the quasiconvexity (or even the convexity) of the initial integrand $W$, it is not possible, in general, to deduce the existence of minimizers for \eqref{invarianthull} because one would need, in addition, the coercivity property for the integrand in \eqref{dosvar} in $\X$ (for every fixed $\F$). Even if we take
$W(\F)=\frac12|\F|^2$ in \eqref{primero}, the corresponding density in \eqref{dosvar} could be bounded from below by a quantity of the form, except for a constant, 
$$
\frac1{\det\X}|\X|^2.
$$
Even if this function is homogeneous of degree zero (and polyconvex \cite{iwmaronn}, see below), it is non-coercive, because there are matrices of norm one and determinant arbitrarily small. Note how we cannot follow the path used for the one-dimensional situation in Section \ref{onedi} based on smoothness, as we would need an independent way to show existence of solutions of the corresponding Euler-Lagrange equation, and then rely on convexity. There is no hope that such a procedure could be successful in a higher-dimensional framework.

\subsection{Another form of invariance}\label{estrella}
As indicated, the variational problem \eqref{invarianthull} for each fixed $\bu_0$ may be difficult to deal with. Yet, despite this fact, the functional defined through that problem turns out to be invariant. 

There is another simpler form to define an invariant density 
$$
W_i(\F):\R^{m\times N}\to\R
$$ 
from a given 
$$
W(\F):\R^{m\times N}\to\R,
$$ 
by simply taking the pointwise minimum in the matrix $\X$, namely
\begin{equation}\label{integrandoi}
W_i(\F)=\inf_{\X\in\R^{N\times N}_+} \det\X\, W\left(\frac1{\det\X}\F\adj\X^T\right),
\end{equation}
where we are putting 
$$
\R^{N\times N}_+=\{\X\in\R^{N\times N}: \det\X>0\}.
$$
Consider the two functionals
\begin{equation}\label{dosfuncionales}
I(\bu)=\int_\Omega W(\nabla \bu(\bbx))\,d\bbx,\quad I_i^*(\bu)=\int_\Omega W_i(\nabla\bu(\bbx))\,d\bbx.
\end{equation}

\begin{proposition}\label{minimo}
This local integral functional $I_i^*$ is invariant for every class $\cD_0$ in Definition \ref{invarianza}, and
$$
I_i^*\le I_i\le I,
$$
where $I_i$ is given in \eqref{invarianthull}. 
\end{proposition}
The proof is elementary and does not require any further comment. What is, on the other hand, a remarkable issue is to decide under what circumstances the invariant hull of an integral functional \eqref{dosfuncionales} is given by $I_i^*$. 
\begin{conjecture}\label{conjetura}
Suppose the functional $I$ is of the typical integral form
$$
I(u)=\int_\Omega W(\nabla\bu(\bbx))\,d\bbx
$$
for a certain continuous integrand 
$$
W(\F):\R^{m\times N}\to\R.
$$ 
If its invariant hull $I_i$ with respect to some invariant class $\cD_0$ turns out to be another integral functional, then $I_i\equiv I_i^*$
$$
I_i(u)=\int_\Omega W_i(\nabla\bu(\bbx))\,d\bbx,
$$
where $W_i$ is given in \eqref{integrandoi}. 
\end{conjecture}

\section{A fundamental two-dimensional example}
To start gaining some insight into high-dimensional examples, we would like to consider the particular initial situation in which
\begin{equation}\label{funcional}
I(\bu)=\frac12\int_\Omega|\nabla\bbu(\bbx)|^2,
\end{equation}
where $\Omega\subset\R^2$ is a model domain (a ball or a box), and maps
$$
\bbu(\bbx)=(u_1(\bbx), u_2(\bbx), u_3(\bbx)):\Omega\to\R^3
$$
belong to the Sobolev space $H^1(\Omega; \R^3)$. Keep in mind that 
$$
\nabla\bu=\begin{pmatrix}\nabla u_1\\\nabla u_2\\\nabla u_3\end{pmatrix}=\begin{pmatrix}\bu_{, 1}&\bu_{, 2}\end{pmatrix},\quad \bu_{, 1}=\frac{\partial\bu}{\partial x}, \bu_{, 2}=\frac{\partial\bu}{\partial y},
$$
and hence, we can recast our quadratic functional in \eqref{funcional} in the two alternative ways
$$
I(\bu)=\frac12\int_\Omega(|\bu_{, 1}(\bbx)|^2+|\bu_{, 2}(\bbx)|^2)\,d\bbx
$$
or
$$
I(\bu)=\frac12\int_\Omega(|\nabla u_1(\bbx)|^2+|\nabla u_2(\bbx)|^2+|\nabla u_3(\bbx)|^2)\,d\bbx.
$$
It is clear that $I$ is not invariant as it is easy to check that it does not fit into the form \eqref{representation}. Hence, 
we would like to calculate its invariant realization 
$$
I_i(\bbu)=\inf_{\Phi\in\cD_0(\Omega)}I(\bbu_\Phi),\quad \bbu_\Phi(\bbx)=\bbu(\Phi^{-1}(\bbx)),
$$
with respect to some suitable invariant class $\cD_0$. 
According to our earlier discussion, we need to face the vector variational problem
\begin{equation}\label{infd}
\hbox{Minimize in }\Phi\in\cD_0(\Omega):\quad \int_\Omega \overline W(\nabla\bbu(\bbx), \nabla\Phi(\bbx))\,d\bbx
\end{equation}
where
$$
\overline W(\F, \X): \R^{3\times2}\times\R^{2\times2} \to\R
$$
is given by
\begin{equation}\label{densidad}
\overline W(\F, \X)=\det\X \,W(\F\X^{-1})=\frac1{2\det\X}|\F\adj \X^T|^2,
\end{equation}
and the map $\bbu(\bbx):\Omega\to\R^3$ is fixed. 

It may be instructive to write things in a fully explicit form to facilitate some calculations. Our variational problem is 
$$
\hbox{Minimize in }\Phi\in\cD_0:\quad \int_\Omega \frac1{2\det\nabla\Phi(\bbx)}|\nabla\bbu(\bbx)\adj\nabla\Phi(\bbx)^T|^2\,d\bbx.
$$
Note that
$$
\nabla\Phi=\begin{pmatrix}\phi_{1, 1}&\phi_{1, 2}\\\phi_{2, 1}&\phi_{2, 2}\end{pmatrix},\quad
\adj\nabla\Phi^T=\begin{pmatrix}\phi_{2, 2}&-\phi_{1, 2}\\-\phi_{2, 1}&\phi_{1, 1}\end{pmatrix}=\begin{pmatrix}\bQ\nabla\phi_2&-\bQ\nabla\phi_1\end{pmatrix}
$$
if
$$
\Phi(\bbx)=(\phi_1(x_1, x_2), \phi_2(x_1, x_2)),\quad \bbx=(x_1, x_2),
$$
and 
$$
\bQ=\begin{pmatrix}0&1\\-1&0\end{pmatrix}, \quad \bQ^T=-\bQ,
$$
is the counterclockwise, $\pi/2$-rotation in the plane. 
We also have
$$
\nabla\bbu(\bbx)=\begin{pmatrix}u_{1, 1}&u_{1, 2}\\u_{2, 1}&u_{2, 2}\\u_{3, 1}&u_{3, 2}\end{pmatrix},\quad \bbu(\bbx)=(u_1(x_1, x_2), u_2(x_1, x_2), u_3(x_1, x_2)).
$$
The product occurring in functional $I(\bbu_\Phi)$ becomes
\begin{align}
\nabla\bbu(\bbx)\adj\nabla\Phi(\bbx)^T=&\begin{pmatrix}u_{1, 1}&u_{1, 2}\\u_{2, 1}&u_{2, 2}\\u_{3, 1}&u_{3, 2}\end{pmatrix}\,\begin{pmatrix}\bQ\nabla\phi_2&-\bQ\nabla\phi_1\end{pmatrix}\nonumber\\
=&\begin{pmatrix}\nabla u_1\cdot\bQ\nabla\phi_2&-\nabla u_1\cdot\bQ\nabla\phi_1\\
\nabla u_2\cdot\bQ\nabla\phi_2&-\nabla u_2\cdot\bQ\nabla\phi_1\\\nabla u_3\cdot\bQ\nabla\phi_2&-\nabla u_3\cdot\bQ\nabla\phi_1\end{pmatrix}.\nonumber
\end{align}
Hence
\begin{equation}\label{explicito}
I(\bbu_\Phi)=\int_\Omega\frac1{\nabla\phi_1\cdot\bQ\nabla\phi_2}\sum_{1\le i\le 3, 1\le j\le2}\frac12(\bQ\nabla u_i\cdot\nabla\phi_j)^2\,d\bbx,
\end{equation}
and we write, for each fixed map $\bbu(\bbx):\Omega\to\R^3$, 
$$
E_\bbu(\Phi)=I(\bbu_\Phi).
$$
If we put
\begin{gather}
\tilde W(\bU, \X)=\frac1{\X_1\cdot\bQ\X_2}\sum_{1\le i\le 3, 1\le j\le2}\frac12(\bU_i\cdot\X_j)^2,\label{densidadi}\\ 
\bU=\begin{pmatrix}\bU_1\\\bU_2\\\bU_3\end{pmatrix}\in\R^{3\times2}, \X=\begin{pmatrix}\X_1\\\X_2\end{pmatrix}\in\R^{2\times2},\nonumber
\end{gather}
then
$$
E_\bbu(\Phi)=\int_\Omega\tilde W(\nabla\bbu(\bx)\bQ^T, \nabla\Phi(\bbx))\,d\bbx.
$$
\begin{lemma}\label{basico}
Let $\F\in\R^{3\times2}$ be given.
\begin{enumerate}
\item  The integrands in \eqref{densidad} or in \eqref{densidadi}, for fixed $\F$, 
$$
\X\in\R^{2\times2}\mapsto\overline W(\F, \X), \tilde W(\F, \X)
$$
are polyconvex over the set $\det\X>0$, and positively homogeneous of degree zero. 
\item The absolute minimum of the function
$$
\X\mapsto \overline W(\F, \X)=\frac1{2\det\X}|\F\adj \X^T|^2
$$
takes place when the matrix
$$
\adj\X\,\F^T\F\,\adj\X^T
$$
is a multiple of the identity (of dimension $2\times2$), and in this case
\begin{equation}\label{igualdad}
\overline W(\F, \X)=|\F_1\wedge\F_2|,\quad \F=\begin{pmatrix}\F_1&\F_2\end{pmatrix}.
\end{equation}
Due to homogeneity, the minimum is taken on for all such multiples.  
\end{enumerate}
\end{lemma}
\begin{proof}
We comment on the integrand $\overline W$. The arguments for $\tilde W$ are exactly the same. 

The polyconvexity is elementary because the real function
$$
w_\F(t, \X):\R^+\times\R^{2\times2}\to\R,\quad w_\F(t, \X)=\frac1{2t}|\F\X|^2
$$
is convex in all of its arguments. Note that
$$
\overline W(\F, \X)=w_\F(\det\X, \adj\X^T).
$$
Check also closely related calculations in \cite{pedregal}. 
The homogeneity is straightforward. The second item is an interesting Multivariate Calculus exercise that can be facilitated by the explicit form in \eqref{densidad}. We will go through a more general calculation below.
\end{proof}

It is quite remarkable that the infimum in \eqref{infd}, coming from \eqref{funcional} can be computed explicitly for every map $\bbu$ for suitable classes $\cD_0$. This is something exclusive of the quadratic functional in \eqref{funcional}. In fact this is a very special situation for which, in the context of Subsection \ref{estrella}, $I_i$ and $I^*_i$, with integrand in \eqref{igualdad} coincide
\begin{theorem}\label{invarianza2}
Fix three different points $\bP_i$, $i=1, 2, 3$, on the boundary $\partial\Omega$ of the unit circle $\Omega\subset\R^2$. Let $\cD_0$ be the subclass of $\cD$ maintaining each $\bP_i$ unchanged. 

The invariant realization of $I$ in \eqref{infd} with respect to $\cD_0$ is the area functional
$$
I_i(\bbu)=\int_\Omega|\bbu_{, 1}(\bbx)\wedge\bbu_{, 2}(\bbx)|\,d\bbx
$$
for each smooth $\bbu(\bbx):\Omega\to\R^3$.
\end{theorem}
\begin{proof}
Because of the second part of Lemma \ref{basico}, the proof amounts to showing that the metric given in the unit disk by the symmetric definite-positive matrix
$$
\begin{pmatrix}|\bbu_{ ,1}|^2&\bbu_{ ,1}\cdot\bbu_{ ,2}\\\bbu_{ ,1}\cdot\bbu_{ ,2}&|\bbu_{ ,2}|^2\end{pmatrix}
$$
can be made rigid, i.e. a multiple of the identity, through a suitable map $\Phi\in\cD_0$. This map $\Phi$ is nothing but a solution of the Beltrami equation
\begin{gather}
\overline\partial\Phi=\mu\,\partial\Phi,\nonumber\\ 
\mu=\frac{|\bbu_{ ,1}|^2-|\bbu_{ ,2}|^2+2i\bbu_{ ,1}\cdot\bbu_{ ,2}}{|\bbu_{ ,1}|^2+|\bbu_{ ,2}|^2+2\sqrt{|\bbu_{ ,1}|^2|\bbu_{ ,2}|^2-(\bbu_{ ,1}\cdot\bbu_{ ,2})^2}}.\nonumber
\end{gather}
In addition, such a solution $\Phi$ can indeed be found complying with the condition on the three points $\bP_i$. All of this is well-known in the theory of quasi-conformal maps. Check \cite{asiwga}, \cite{iwon2}, \cite{iwmaronn}. In particular, the main theorem in \cite{shibata} is also enlightening. 
\end{proof}

\begin{corollary}[Minimal surfaces]\label{minimall}
Let $\cD_0$ be an invariant class of those determined in the previous result, and let $\cU_0\subset H^1(\Omega; \R^2)$ be compatible with $\cD_0$ as in Definition \ref{invarianza}. Every minimizer $\bbu_0\in\cU_0$ of the problem
\begin{equation}\label{primitivo}
\hbox{Minimize in }\bbu\in\cU_0:\quad \int_\Omega\frac12|\nabla\bbu(\bbx)|^2\,d\bbx
\end{equation}
will be a minimizer for
$$
\hbox{Minimize in }\bbu\in\cU_0:\quad \int_\Omega|\bbu_{, 1}(\bbx)\wedge\bbu_{ , 2}(\bbx)|\,d\bbx.
$$
\end{corollary}
This corollary is a direct consequence of Proposition \ref{existencia} after Theorem \ref{invarianza2}. There are two additional important facts for the classic problem of the minimal surfaces to be fully and satisfactorily solved. The first is to ensure that problem \eqref{primitivo} admits one minimizer. This is classical and pretty accesible. It requires, however, to deal with boundary conditions other than a typical Dirichlet datum around $\partial\Omega$ as in the classical three-point boundary condition in Theorem \ref{invarianza2}. 
The second one is much more delicate and is concerned with showing that such minimizer $\bbu_0$ represents a regular surface, that is to say the two vectors $\bbu_{0 , 1}$, $\bbu_{0 , 2}$ are independent for all $\bbx\in\Omega$. Plateau's problem is, however, much, much more. See survey \cite{harrisonpugh}, and notes \cite{schmidt} among an overwhelming amount of literature on this problem. 

\begin{remark}
It is interesting to notice that Proposition \ref{intermedio}, together with Theorem \ref{invarianza2}, permit to conclude that the invariant realization of the functional
$$
E(\bbu)=\int_\Omega \left|\frac{\partial\bbu}{\partial x_1}\right|\,\left|\frac{\partial\bbu}{\partial x_2}\right|\,d\bbx
$$
with respect to a class $\cD_0$ as in Theorem \ref{invarianza2} is given again by the area functional
$$
I_i(\bbu)=\int_\Omega|\bbu_{, 1}(\bbx)\wedge\bbu_{, 2}(\bbx)|\,d\bbx, 
$$
and, in particular according to Corollary \ref{minimall}, the minimizer $\bbu_0\in\cU_0$ will also be a minimizer for $E$ in $\cU_0$, despite the fact that the integrand for $E$ is not quasiconvex, nor coercive (see \cite{pedregalN}). 
\end{remark}

The previous analysis can hardly be extended to other functionals, even for simple fully explicit examples.
Suppose we start with an integral functional 
\begin{equation}\label{general}
\int_\Omega W(\nabla\bbu(\bbx))\,d\bbx
\end{equation}
for an integrand $W(\F):\R^{3\times 2}\to\R$ positively-homogeneous of degree two. Let $W_i$ be the integrand defined in \eqref{integrandoi}.  
\begin{problem}
Under what circumstances, is it true that the invariant realization of \eqref{general} is given by
$$
\int_\Omega W_i(\nabla\bbu(\bbx))\,d\bbx\hbox{?}
$$
Can examples be found for which this is not the case?
\end{problem}

\section{The higher-dimensional version}
Multidimensional examples are, as one would expect, much harder to treat. We restrict attention to the fundamental example of minimizing the $N$-volume of parameterized $N$-dimensional manifolds in $\R^m$, $m>N$. 

We want to examine the particular important case for the densities
\begin{gather}
W^N(\F)=\frac1{N^{N/2}}|\F|^N, \quad W^N_i(\F)=\sqrt{\det(\F^T\F)},\nonumber\\ 
\F=\begin{pmatrix} \F_1&\F_2&\dots&\F_N\end{pmatrix}\in\R^{m\times N},\nonumber
\end{gather}
and corresponding functionals
$$
I_N(\bbu)=\int_\Omega W^N(\nabla\bbu(\bbx))\,d\bbx,\quad V_N(\bbu)=\int_\Omega W^N_i(\nabla\bbu(\bbx))\,d\bbx
$$
for regular parameterizations 
$$
\bbu(\bbx)=(u_1(\bbx), u_2(\bbx), \dots, u_m(\bbx)):\Omega\subset\R^N\to\R^m
$$
of a piece of a $N$-manifold in $\R^m$. As before, $\Omega$ is a model domain in $\R^N$, and we can write
$$
\nabla\bbu=\begin{pmatrix}\nabla u_1\\\nabla u_2\\\dots\\\nabla u_m\end{pmatrix}=\begin{pmatrix}\bu_{, 1}&\bu_{, 2}&\dots&\bu_{, N}\end{pmatrix},\quad \bu_{, k}=\frac{\partial\bu}{\partial x_k}.
$$
In particular, 
$$
W^N(\F)=\left(\frac1N\sum_{k=1}^N|\F_k|^2\right)^{N/2}=\left(\frac1N\sum_{l=1}^m|\bbf_l|^2\right)^{N/2}
$$
if
$$
\F=\begin{pmatrix}\F_1&\F_2&\dots&\F_N\end{pmatrix}=\begin{pmatrix}\bbf_1\\\bbf_2\\\dots\\\bbf_m\end{pmatrix}.
$$
The integral
\begin{equation}\label{volumen}
V_N(\bbu)=\int_\Omega \sqrt{\det(\nabla\bbu(\bbx)^T\nabla\bbu(\bbx))}\,d\bbx
\end{equation}
yields the $N$-dimensional volume of the image $\bbu(\Omega)$ in $\R^m$, and, hence,  it is definitely an invariant functional as we will check just below. A result like Theorem \ref{invarianza2} is however not possible in the higher-dimensional case $N>2$. Recall that
\begin{equation}\label{densidadN}
\overline W^N(\F, \X)=\det\X \frac1{N^{N/2}} \left|\frac1{\det\X}\F\adj\X^T\right|^N=\frac1{N^{N/2}\det\X^{N-1}}\left|\F\adj\X^T\right|^N.
\end{equation}
\begin{lemma}\label{basicoN}
Let $\F\in\R^{m\times N}$ be given.
\begin{enumerate}
\item  The integrand in \eqref{densidadN} 
$$
\X\in\R^{N\times N}\mapsto\overline W^N(\F, \X), 
$$
is polyconvex over the set $\det\F>0$, and positively homogeneous of degree zero. 
\item The absolute minimum of the function
$$
\X\mapsto \overline W^N(\F, \X)
$$
takes place when the matrix
$$
\adj\X\,\F^T\F\adj\X^T
$$
is a multiple of the identity (of dimension $N\times N$). In this case, 
$$
\overline W^N(\F, \X)=W^N_i(\F).
$$
Due to homogeneity, the minimum is taken on for all such multiples.  
\end{enumerate}
\end{lemma}
\begin{proof}
Except for a constant to be determined, we write
$$
\overline W(\F, \X)=\det\X|\F\X^{-1}|^N,
$$
or even better for computations
$$
\overline W(\F, \X)=\det\X|\A(\X)|^N,\quad \A(\X)\X=\F.
$$
Here $\A=\A(\X)$ and $\F$ are $m\times N$-matrices, while $\X$ is a non-singular $N\times N$-matrix. In particular, by differentiation, 
$$
\bD\A(\X)\X+\A(\X)=\bcero.
$$
From this equation, we find that
\begin{gather}
\A(\X)^T\bD\A(\X)=-\A(\X)^T\A(\X) \X^{-1},\nonumber\\
\bD\A(\X)^T\A(\X)=-\X^{-T}\A(\X)^T\A(\X).\nonumber
\end{gather}
Since
$$
\A(\X)=\F\X^{-1}=\frac1{\det\X}\F\adj\X^T,
$$
we can write
\begin{equation}\label{formulae}
\det\X^3\,\bD\A(\X)^T\A(\X)=-\adj\X\adj\X\,\F^T\F\adj\X^T.
\end{equation}
On the other hand, 
$$
\bD_\X\overline W(\F, \X)=\adj\X|\A(\X)|^N+N\det\X|\A(\X)|^{N-2}\bD\A(\X)^T\A(\X),
$$
and \eqref{formulae} carries us, for a critical matrix $\X$, to 
$$
\det\X^2|\A(\X)|^2\adj\X-N|\A(\X)|^{N-2}\adj\X\adj\X\,\F^T\F\adj\X^T=\bcero.
$$
Since $\adj\X$ is not singular, this is our claim.  Indeed, if we take determinant in this last equation, we see that
$$
(\det\X)^{2N}|\A(\X)|^{2N}=N^N(\det\X)^{2(N-1)}\det(\F^T\F),
$$
or
$$
(\det\X)^2|\A(\X)|^{2N}=N^N\det(\F^T\F), 
$$
that is to say
$$
\overline W(\F, \X)=N^{N/2}\sqrt{\det(\F^T\F)}.
$$
\end{proof}

From Proposition \ref{minimo}, all we can conclude is the following.
\begin{corollary}
The integrand
$$
V_N(\bbu)=\int_\Omega W_i^N(\nabla\bbu(\bbx))\,d\bbx
$$
given in \eqref{volumen},
is invariant for every invariant subclass $\cD_0\subset\cD$, and
$$
V_N(\bbu)\le I_i(\bbu)\le I_N(\bbu).
$$
\end{corollary}
Concerning the possibility that $V_N=I_i$, it looks unlikely to be so given the classical Liouville theorem restricting the solutions of the equation
$$
\nabla\Phi(\bbx)^T\nabla\Phi(\bbx)=\det(\nabla\Phi(\bbx))^{2/N}\G(\bbx),
$$
for a suitable symmetric, matrix-valued mapping 
$$
\G(\bbx):\Omega\to\bbS(N),\quad \det\G(\bbx)=1,
$$ 
to Möbius transformations (\cite{iwmar}). Note that such a mapping $\G$ complying with the previous equation comes directly from 
$$
\nabla\bbu(\bbx)^T\nabla\bbu(\bbx)=\mu(\bbx)(\det\nabla\Phi(\bbx))^{2(N-1)}\nabla\Phi(\bbx)^T\nabla\Phi(\bbx),
$$
which is the basic functional equation in the proof of Lemma \ref{basicoN}. 
\begin{problem}
For $N>2$, find an alternative density  
$$
\tilde W^N(\F)\to\R,\quad \F\in\R^{m\times N},
$$
positively homogeneous of degree $N$ such that
$$
I_i(\bbu)=V_N(\bbu)
$$
for smooth mappings $\bbu:\Omega\subset\R^N\to\R^m$.
\end{problem}
If Conjecture \ref{conjetura} turns out to be correct, all that is required is to design $\tilde W$ in such a way that the corresponding invariant realization $I_i$ is a local, integral functional. 

We finish with another problem whose solution seems out of reach. 

\begin{problem}
Consider the power integral functional
$$
I_p(\bbu)=\int_\Omega |\nabla\bbu(\bbx)|^p\,d\bbx,\quad p>0,
$$
over some specified invariant class $\cD_0\subset\cD$. For what ranges of the exponent $p$, is the invariant hull of $I_p$ trivial, or of non-local, integral form?
\end{problem}

\end{document}